\documentclass[12pt]{article}
\usepackage{amsmath, amsfonts, amsthm, amssymb, verbatim}
\usepackage{graphicx}
\usepackage[mathcal]{euscript}
\usepackage{amstext}
\usepackage{amssymb,mathrsfs}
\usepackage{bbm}
\usepackage[latin1]{inputenc}
\newcommand{\R}{\mathbb R}

  \newcommand{\E}{\mathbb E}

\newcommand{\PP}{\mathbb P}
\newcommand{\QQ}{\mathbb Q}
\newcommand{\calX} {\ensuremath {\mathcal{X}}}
\newcommand{\calC} {\ensuremath {\mathcal{C}}}
\newcommand{\calY} {\ensuremath {\mathcal{Y}}}
\newcommand{\calR} {\ensuremath {\mathcal{R}}}
\newcommand{\calF} {\ensuremath {\mathcal{F}}}

\newcommand{\vp}{\varphi}

\newcommand \loc    {\text{loc}}

\newcommand{\dive}{{\rm div}}

\newcommand{\indicator}[1]{\mathbbm{1}_{\left\{ {#1} \right\} }}
\newcommand{\clg}[1]{{\mathcal{#1}}}

\newtheorem{theorem}{Theorem}[section]

 \newtheorem{remark}[theorem]{Remark}
\newtheorem{lemma}[theorem]{Lemma}

\newtheorem{definition}[theorem]{Definition}
\newtheorem{hypothesis}[theorem]{Hypothesis}

\begin{document}
\title{On a class of stochastic transport  equations \\
for $L_{loc}^2$ vector fields}

\author{Ennio Fedrizzi\footnote{Universit\'e de Lyon, CNRS UMR 5208, 
Universit\'e Lyon 1, Institut Camille Jordan, France. 
E-mail: {\sl fedrizzi@math.univ-lyon1.fr}.
},
Wladimir Neves\footnote{Instituto de Matem\'atica, Universidade Federal
do Rio de Janeiro, Brazil. 
E-mail: {\sl wladimir@im.ufrj.br}.
}, 
Christian Olivera\footnote{Departamento de Matem\'{a}tica, Universidade Estadual de Campinas, Brazil. 
E-mail:  {\sl  colivera@ime.unicamp.br}.
}}

\date{}

\maketitle





\vspace{0.3cm} \noindent {\bf MSC2010 subject classification:} 60H15, 
 35R60, 
 35F10, 
 60H30. 

%
\begin{abstract}
We study in this article the existence and uniqueness of solutions to a class of 
stochastic transport equations with irregular coefficients. Asking only boundedness 
of the divergence of the coefficients (a classical condition in both the deterministic and stochastic setting), 
we can lower the integrability regularity required in known results on the coefficients 
themselves and on the initial condition, and still prove uniqueness of solutions.
 \end{abstract}

\begin{abstract}
{\bf R\'esum\'e:} On \'etudie existence et unicit\'e des solutions de l'\'equation du transport stochastique avec des coefficients tr\`es irr\'eguliers. Si on suppose que la divergence soit born\'e (une condition classique \`a la fois dans le cadre d\'eterministe et stochastique), on peut arriver \`a montrer l'unicit\'e m\^eme avec des coefficients et des donn\'ees seulement localement int\'egrables.
 \end{abstract}

%

\maketitle

%
\section {Introduction} \label{Intro}

The linear transport equation, that is 
\begin{equation}\label{trasports}
    \partial_t u(t, x) +  b(t,x) \cdot  \nabla u(t,x)  = 0 \, ,
\end{equation}
has several and diverse physical applications, for instance related to fluid dynamics
as it is well described in Lions' book \cite{lion1,lion2}. See also Dafermos' book \cite{Dafermos} 
for more general applications of the transport equation in the domain of conservation laws.

\bigskip
In view of applications to multiphase flows through porous media, 
we are interested to study this equation (and in particular the uniqueness property) without Sobolev, or even BV spatial regularity 
of the drift vector field $b(t,x)$. This type of problems is addressed in   
 \cite{NCWN1}, \cite{NCWN2}, \cite{Consta},  \cite{Consta2},  and it is one of the  motivations to consider the vector field $b$ with just $L^2_\loc$ regularity. However, with such low regularity of the coefficient there is no hope to obtain uniqueness results for the above transport equation, due to the counter example provided by M. Aizeman \cite{Aizem}. 
 
Still, hope remains if we consider a stochastic version of the transport equation: we show that with the introduction of a (even very small) random perturbation in the equation it is possible to obtain uniqueness in a suitable, quite general class of solutions. This is the main contribution of this work. 

Our results appear to be well-adapted to the study of the so called Stochastic Muskat Problem, and could constitute a first essential step towards the solution of this important and hard problem. This inaugural type of perturbation of the original Muskat problem may open new research directions, with applications in particular to numerical simulations related to the planning and operation of oil industry.  In the last section we will further discuss these motivations and provide some more details on the Stochastic Muskat Problem.

\bigskip
Let us now briefly recall some of the main recent 
results concerning the transport equation.
In 1989, R. DiPerna and P.L. Lions \cite{DL} proved that $W^{1,1}$ 
spatial regularity of the vector field $b(t,x)$
(together with a condition of boundedness on the divergence) is enough to ensure uniqueness of weak solutions. 
In 1998, P.L. Lions introduced in \cite{lion98}  the so-called piecewise $W^{1,1}$ class and extended the results of \cite{DL} for this type
of regularity. 
Last but not least, in 2004, L. Ambrosio \cite{ambrisio} proved uniqueness for $BV_{loc}$ vector fields. 
It is also worth mentioning the works of M. Hauray \cite{MH}, and G. Alberti, S. Bianchini, G. Crippa 
\cite{ABC} both in 2 dimensions, where the drift 
does not have any differentiability regularity, but with some additional geometrical conditions.
We would also like to mention the generalizations to transport-diffusion
equations and the associated stochastic differential equations by C. Le Bris
and P.L. Lions \cite{LBL04, BrisLion}  and A. Figalli \cite{Figalli}.

\bigskip
Recently, much attention has been devoted to extensions of this theory under 
random perturbations of the drift vector field, namely considering the following stochastic linear transport equation (SLTE)
\begin{equation}\label{trasport}
 \left \{
\begin{aligned}
    &\partial_t u(t, x,\omega) + \big(b(t, x) + \frac{d B_{t}}{dt}(\omega)\big ) \cdot  \nabla u(t, x,\omega) = 0 \, ,
    \\[5pt]
    &u|_{t=0}=  u_{0} \, .
\end{aligned}
\right .
\end{equation}
Here, $(t,x) \in [0,T] \times \R^d$, $\omega \in \Omega$ is an element of the probability space $(\Omega, \PP, \calF)$, $b:\R_+ \times \R^d \to \R^d$ is a given vector field and $B_{t} = (B_{t}^{1},...,B _{t}^{d} )$ is a
standard Brownian motion in $\mathbb{R}^{d}$. The stochastic
integration is to be understood in the Stratonovich sense. 

Most results can be extended to transport equations defined for $(t,x) \in [0,T] \times U$, where the domain $U$ may be the torus $\Pi^d$ or a bounded open (regular) subset of $\R^d$, which is the most interesting case for applications. In the latter case it is assumed that $b$ is tangent to $\partial U$ (in a suitable trace sense), while in the case where the full space is considered ($U=\R^d$), some additional growth conditions are usually required on $b$.

\bigskip
A very interesting situations is when the stochastic problem is better behaved than the deterministic one. A first result in this direction was 
given by F. Flandoli, M. Gubinelli and  E. Priola
in \cite{FGP2}, where they obtained wellposedness of the stochastic problem for an H\"older continuous drift term, with some integrability conditions on the divergence. Their driving motivation was the analysis of the gain in regularity, due to the noisy perturbation, with respect to the deterministic problem. Their approach is based on a careful analysis of the characteristics. Using a similar approach, in \cite{Fre1} a wellposedness result is obtained under only some integrability conditions on the drift, with no assumption on the divergence, but for fairly regular initial conditions. There, it is only assumed that
\begin{equation}\label{LPSC}
    \begin{aligned}
    &b\in L^{q}\big( [0,T] ; L^{p}(\mathbb{R}^{d}) \big) \, , \\[5pt]
 \mathrm{for}  \qquad &\  p,q \in [2,\infty) \, , \qquad   \qquad  \frac{d}{p} + \frac{2}{q} < 1 \, .
\end{aligned}
\end{equation}
In fact, this condition (with local integrability)
was first considered by  Krylov and R\"{o}ckner in \cite{Krylov}, where they 
proved the existence and uniqueness
of strong solutions for the SDE (the equation of characteristics for the SLTE)
\begin{equation}\label{itoass}
X_{s,t}(x)= x + \int_{s}^{t}   b(r, X_{s,r}(x)) \ dr  +  B_{t}-B_{s} \, ,
\end{equation}
such that
$$
 \mathbb{P}\Big( \int_0^T |b(t,X_t)| \ dt= \infty \Big)= 0 \, .
$$
It is interesting to remark that condition \eqref{LPSC} (more precisely with also equality) is known as the Ladyzhenskaya-Prodi-Serrin condition in the fluid dynamics literature.

This approach based on stochastic characteristics proved to be quite efficient to prove existence, uniqueness and regularity of solutions of the stochastic transport equation. It has produced interesting results on strong uniqueness \cite{FGP2, Fre1} of (regular) solutions and weak uniqueness of (less regular) solutions \cite{LBL04}. However, it has some limitations, as one has to be able to solve the equation of characteristics \eqref{itoass}. This can be done working with regularized coefficients, as in do Section 3 below to prove existence of solutions, and then passing to the limit. As mentioned above, the limit equation can be given a meaning when the drift coefficient is in the Krylov-R\"ockner class \eqref{LPSC}, but certainly not in $L^2_{loc}$. Therefore, there is little hope to obtain even weak uniqueness with this approach in the case of $L^2_{loc}$ coefficients: this was already remarked by \cite{LBL04}.

The wellposedness of the Cauchy problem \eqref{trasport} under condition \eqref{LPSC}
for measurable and bounded initial data was considered also in \cite{NO}. In that paper 
the authors are not interested in the regularizing effects on the 
solution due to the noise, since they consider (possibly) discontinuous 
solutions, which are often the relevant ones for physical applications, 
see also \cite{WNCOBD}.

Later, in \cite{Beck}, using a technique based on the regularizing effect 
observed on expected values of moments of the solution, wellposedness of \eqref{trasport} was also obtained for the limit cases of $p,q=\infty$ or when the 
inequality in \eqref{LPSC} becomes an equality.

We mention that other approaches have also been
used to study stochastic linear transport equations. For example, in \cite{Maurelli} the Wiener chaos decomposition  is employed to deal with a weakly differentiable drift, 
or, in \cite{MNP14}, Malliavin calculus, which allows to deal with just a bounded drift term.
However, all these methods seem to have problems in dealing with nonlinear equations, and 
the interesting question of the improvement of the theory due to introduction of noise
for nonlinear equations still remains largely open. The situation is quite delicate: very few results are known, and easy counterexamples can also be constructed. We address the reader to  \cite{Flanlect} for a more detailed discussion of this topic, and only report here the observation that a 
multiplicative noise as the one used in the SLTE is not enough to improve 
the regularity of solutions of the following stochastic Burgers equation
$$
    \partial_t u(t, x,\omega) +  \partial_x u(t, x,\omega)  \big( u(t,x,\omega) + \frac{d B_{t}}{dt}(\omega)\big ) = 0 \, .
$$
Indeed, for this equation one can observe the appearance of shocks in finite time, just as for the deterministic Burgers equation. 
For a different approach related to stochastic scalar conservation laws, we address the reader to \cite{lions}. \\

The main issue of this paper it to prove uniqueness of weak solutions for $L_{loc}^{2}$ vector fields
(an intrinsically stochastic result as mentioned before) for measurable bounded initial data. Since we are not using characteristics to prove uniqueness, the integrability hypothesis we need on the vector field $b$ are less restrictive than \eqref{LPSC}. However, we ask that the divergence of $b$ be bounded and some regularity in mean for the solution of the stochastic transport equation. 
We stress that the approach presented here, though inspired by the above mentioned works, remains quite different as our proof of the uniqueness property relies on properties of the stochastic exponentials. This seems to be the first time stochastic exponentials are used to prove uniqueness for a SPDE.

It is well known that the expected value $U=\E[u]$ of any solution $u$ of the stochastic transport equation \eqref{trasport} solves a parabolic equation (sometimes called viscous transport equation)
\begin{equation}\label{PDE-intro}
\partial_t U(t,x) +  b(t,x) \cdot \nabla U(t,x)= \frac{1}{2} \Delta U(t,x) \, ,
\end{equation}
which enjoys very good regularization and uniqueness properties: this is ultimately due to the 
passage from the Stratonovich stochastic integral to the It\^o formulation, and was used for example in \cite{Beck}. It is therefore possible to obtain uniqueness in law even with irregular coefficients, see for example the result of \cite{LBL04}, where weak uniqueness is obtained in the same setting we will use.

One way to obtain strong uniqueness using this link with a parabolic PDE is to consider renormalized solutions $\beta(u)$, as introduced by Di Perna and Lions \cite{DL}. This requires however at least BV regularity of the drift coefficient $b$, plus bounded divergence and linear growth. See \cite{ambrisio} for the deterministic case, \cite{AttFl11} for the stochastic case. The interesting result contained in the latter is that in the stochastic setting one can relax the condition on bounded divergence to allow for a drift with a component having bounded divergence and linear growth, plus a bounded component.

Contrary to the above examples, to obtain a stronger form of uniqueness with our approach we are brought to consider a {\it family} of parabolic equations
\begin{equation}\label{para}
\partial_t V(t,x) + \big( b(t,x) + h(t) \big) \cdot \nabla V(t,x)= \frac{1}{2} \Delta V(t,x)
\end{equation}
with $h(t)\in L^2(0,T)$. These equations are similar to the Fokker-Planck equation \eqref{PDE-intro} studied for example in \cite{LBL04} and \cite{BrisLion}. In particular, their results provide wellposedness for the family of parabolic equations \eqref{para} in the space $L^\infty([0,T] ; L^1 \cap L^\infty(\R^d)) \cap L^2([0,T] ; H^1(\R^d))$ (see \cite[Proposition 5.4]{LBL04}). Here, we will work with a very similar space, $\calC^0(L^2) \cap L^2( H^1) $, see Definition \ref{defisolu}.  We will show (see Lemma \ref{lem:eq V} ) that for a solution $u$ to the stochastic transport equation \eqref{trasport}, its expected value $V=\E[uF]$ against any stochastic exponential $F$ solves, as soon as it is sufficiently regular, a parabolic equation of the family \eqref{para}, and therefore, as one could then expect, is unique.  Using the uniqueness result not only for a single equation but for the whole family \eqref{para}, and looking at stochastic exponentials as test functions (they form a family which is large enough), we are able to obtain almost sure uniqueness.


We also stress that our uniqueness result is established in the class of quasiregular weak solutions
(see Definition \ref{defisolu}); this class encompass the natural one, containing 
solutions obtained by regularization processes, see Remark \ref{REMAS}.  
Uniqueness in this class could be used to apply a fixed-point argument and show existence of solutions to
the Stochastic Muskat Problem, see the discussion in Section \ref{SMP}.

This paper is organized as follows. In the next section we present our setting, introduce some notation and define the class of quasiregular weak solutions. In Section \ref{EXISTENCE} we prove existence of such solutions. The main result, uniqueness in the class considered, is contained in Section \ref{UNIQUE}. In section \ref{SMP} we present the Stochastic Muskat problem, one of the motivations that drove us to consider this problem. To ease the presentation, the proofs of some technical results are postponed to the Appendix.

\section{Definition of  weak solutions}

We present now the setting and a suitable definition of weak solutions to equation \eqref{trasport}, adapted to treat the problem of well-posedness 
under our very weak assumptions on the regularity of the coefficients and the initial condition. On the drift coefficient $b$ we shall only assume local integrability and a mild growth control condition. 
Its divergence is assumed to be bounded in space and  integrable in time.

\begin{hypothesis}\label{hyp} 
We shall always assume that the vector field $b$ satisfies:
\begin{equation}\label{con1}
   b \in L^2_\loc\big( [0,T] \times \R^d  \big) 
\end{equation} 
and 
\begin{equation}\label{con2}
 \dive  b(t,x) \in L^1 \big( [0,T] ; L^\infty ( \R^d) \big) \, . 
\end{equation}
Moreover, the initial condition is taken to be
\begin{equation*}\label{conIC}
 u_0 \in  L^2(\R^d)  \cap L^\infty (\R^d) \, . 
\end{equation*}
\end{hypothesis}

This first set of hypothesis is sufficient to prove existence of solutions.
However, as in the classical deterministic setting, to obtain 
uniqueness an additional hypothesis on the growth of the 
drift coefficient is needed. 

\begin{hypothesis} \label{hyp2}
Assume that the vector field $b$ satisfies:
\begin{equation}\label{con3}
 There \ exists \ R> 0, \ such \ that \   \frac{b(t,x)}{1+|x|} \in L^1\big([0,T]; L^\infty(\R^d-B_R) \big),\\
\end{equation}
where \ $B_R= \left\{  x \in \R^{d} : \ |x|\leq R \right\}$.
\end{hypothesis}


\medskip
We shall work on a fixed time interval $t\in[0,T]$, and throughout the paper 
we will use a given probability space $(\Omega, \PP, \calF)$, on which there exists 
an $\R^d$-valued Brownian motion $B_t$ for $t\in[0,T]$. We will use the 
natural filtration of the Brownian motion $\calF_t = \calF_t^B$, 
and restrict ourselves to consider the collection of measurable sets given by the $\sigma$-algebra $\calF = \calF_T$, augmented by the $\PP$-negligible sets. 
Moreover, for convenience 
we introduce the following set of random variables, called the space of stochastic exponentials:
$$ 
  \calX := \Big\{ F= \exp \Big( \int_0^T h(s) \cdot d B_s - \frac{1}{2}\int_0^T |h(s)|^2 \, ds \Big)   \    \Big|     \    h \in L^2\big([0,T] ;\R^d \big) \Big\} \, .
$$
Further details on stochastic exponentials and some useful properties are collected in the Appendix.
In particular,
the technical assumption that the $\sigma$-algebra we are using is 
the one provided by the Brownian motion is essential to ensure that 
the family of stochastic exponentials provides a set of test functions 
large enough to obtain almost sure uniqueness.

\medskip
The next definition tells us in which sense a stochastic process is a weak solution of \eqref{trasport}. 
Hereupon, we will use the summation convention on repeated indices.

\begin{definition}\label{defisolu}  A stochastic process $u\in L^2 \cap L^\infty \big( \Omega \times[0,T] \times \R^d \big)$ is called 
a quasiregular weak solution of the Cauchy problem \eqref{trasport} when:

\begin{itemize}
\item ({\it Weak solution}) For any $\varphi \in C_c^{\infty}(\R^d)$, the real valued process $\int  u(t,
  x)\varphi(x)  dx$ has a continuous modification which is an
$\mathcal{F}_{t}$-semimartingale, and for all $t \in [0,T]$, we have $\mathbb{P}$-almost surely
\begin{equation} \label{DISTINTSTR}
\begin{aligned}
    \int_{\R^d} u(t,x) \varphi(x) dx = &\int_{\R^d} u_{0}(x) \varphi(x) \ dx
   +\int_{0}^{t} \!\! \int_{\R^d} u(s,x) \ b^i(s,x) \partial_{i} \varphi(x) \ dx ds
\\[5pt]
    &     +   \int_{0}^{t} \!\! \int_{\R^d} \dive b(s,x) u(s,x) \varphi(x) \ dx \, ds 
\\[5pt]    
    &+ \int_{0}^{t} \!\! \int_{\R^d} u(s,x) \ \partial_{i} \varphi(x) \ dx \, {\circ}{dB^i_s} \, .
\end{aligned}
\end{equation}

\item ({\it Regularity in Mean}) For each function $F \in \calX$, the deterministic function $V:=\mathbb{E}[uF]$ is a measurable bounded function, 
which belongs to  $L^{2}([0,T]; H^{1}(\R^d) ) \cap C([0,T]; L^2(\R^d))$.
\end{itemize}
\end{definition}

If the Stratonovich formulation, in view of the Wong-Zakai approximation theorem, is often considered to be the ``natural'' one for this kind of problems, it is useful for computations to present also the Ito formulation of equation \eqref{DISTINTSTR}. It reads 
\begin{align}\label{Ito-weak}
    \int_{\R^d} u(t,x) \varphi(x) \, dx =  &  \int_{\R^d} u_0(x) \varphi(x) \, dx    \\
    & + \int_{0}^{t} \!\! \int_{\R^d} u(s,x) \, \big( b(s,x) \cdot \nabla \varphi(x) + \varphi(x)   \dive b(s,x) \big) \, dx ds     \nonumber   \\
    & + \int_{0}^{t}  \Big( \int_{\R^d} u(s,x) \ \nabla \varphi(x) \ dx \Big) \cdot dB_s  \nonumber  \\
    & + \frac{1}{2}  \int_{0}^{t} \!\! \int_{\R^d} u(s,x) \ \Delta \varphi(x) \, dx ds \, .   \nonumber
\end{align}

\begin{remark}\label{rem:filtration}
Let us stress that the notion of solution in the above definition is ``strong" in probabilistic sense, since the Brownian motion is a priori given. However, quasiregular solutions are processes which, integrated against smooth test functions in space, are semimartingales only with respect to the \it{Brownian filtration}, not an arbitrarily chosen filtration.
\end{remark}

\begin{remark}\label{rem:regularity b}
The condition of regularity in mean in the above definition is introduced to replace some stronger regularity assumptions on $b$, usually necessary to prove uniqueness with more traditional approaches. Remark however that, as soon as the drift coefficient is a little bit more regular, regularity in mean is no longer necessary: if \eqref{con1} is replaced by
\begin{equation*}
b \in L^1_\loc\big( [0,T] ; W^{1,1}_{loc} ( \R^d)  \big) \, ,
\end{equation*}
then by \cite{LBL04} we have weak uniqueness for the stochastic transport equation \eqref{trasport}. Moreover, under the same regularity condition, it was recently shown \cite{CO13} existence and strong (in probabilistic sense) uniqueness of $L^p$ weak (in PDE sense) solutions. Observe that the $(\omega,t,x)$-a.e. uniqueness obtained in the present work is implied by strong uniqueness.
\end{remark}

\begin{lemma}\label{lem:eq V}
If $u$ is a quasiregular weak solution of \eqref{trasport}, then for each function $F \in \calX$, the deterministic function $V:=\mathbb{E}[uF]$ satisfies  the parabolic equation \eqref{para} in the weak sense, with initial condition given by $V_0 = u_0$.
\end{lemma}

\begin{proof}
Take any $F\in \calX $ and any quasiregular weak solution $u$. By definition, $V(t,x) \in L^{2} \big( [0,T]; H^{1}(\R^d) \big) \cap C \big( [0,T] ; L^2 (\R^d) \big)$. Consider the It\^o integral form of the equation satisfied by $u$, as given in \eqref{Ito-weak}. To obtain an equation for $V$ we multiply this equation by $F$ and take expectations:
\begin{align}\label{eq1}
    \int_{\R^d} V(t,x) \varphi(x) \, dx = & \int_{\R^d} V(0,x) \varphi(x) \, dx \nonumber \\
    & + \int_{0}^{t} \!\! \int_{\R^d} V(s,x) \, \big( b(s,x) \cdot \nabla \varphi(x) + \varphi(x)   \dive b(s,x) \big) \, dx ds      \nonumber  \\[5pt]
    & + \E \Big[  \int_{0}^{t}  \Big( \int_{\R^d} u(s,x) \ \nabla \varphi(x) \ dx \Big) \cdot dB_s  \ F \Big] \\[5pt]  
    & + \frac{1}{2}  \int_{0}^{t} \!\! \int_{\R^d} V(s,x) \ \Delta \varphi(x) \, dx ds \, .   \nonumber
\end{align} 
By definition of quasiregular weak solutions, $\int_{\R^d} u(\cdot,x) \varphi(x) \, dx$ is an adapted square integrable process for any $\varphi\in C^\infty_c(\R^d)$. Therefore,
$$Y_s = \int_{\R^d} u(s,x) \nabla \varphi(x) \, dx$$
is also an adapted square integrable process. The expected value of the stochastic integral on the second line of \eqref{eq1} can be rewritten as the expected value of a Lebesgue integral against a certain function $h \in L^2\big([0,T] \big)$ due to the properties of stochastic exponentials:
\begin{align*}
 \E \Big[  \int_{0}^{t}  \Big( \int_{\R^d} u(s,x) \ \nabla \varphi(x) \ dx \Big) \cdot dB_s  \ F \Big] =  \int_{0}^{t} \int_{\R^d} \!  V(s,x) h(s) \cdot \nabla \varphi(x) \, dx ds   \, .
\end{align*}
 This is shown in detail in Lemma \ref{lemma B-F} in the Appendix.
\medskip

Now, due to the regularity of $V$, we see that $V$ is a weak solution of the PDE \eqref{para}, that is to say, for each test function $\vp \in C^\infty_c(\R^d)$ 
\begin{align}\label{eq V}
    \int_{\R^d} V(t,x)  \varphi (x) \, dx =  & \int_{\R^d} V(0,x)  \varphi (x) \, dx   \nonumber  \\
     &+  \int_{0}^{t} \int_{\R^d} \!  V(s,x) \big(   b(s,x) \cdot  \nabla \varphi (x) +   \varphi (x) \dive b(s,x)   \big)  \, dx  ds     \nonumber    \\[5pt]
     &+  \int_{0}^{t} \int_{\R^d} \!  V(s,x) h(s) \cdot \nabla \varphi(x) \, dx ds     \nonumber    \\[5pt]
     & - \frac{1}{2}\int_{0}^{t}   \int_{\R^d} \!   \nabla V(s,x)  \cdot \nabla \varphi (x) \, dx  ds  \, .
\end{align}

As explained in the Appendix, due to the properties of stochastic exponentials we have that $F$ is a martingale with mean $1$. Since $u_0$ is deterministic, it immediately follows that $V_0 = \E \big[ u_0 F \big] = u_0$.
\end{proof}

\begin{remark}\label{rem:Girsanov}
Let us spend a few words to discuss one of the reasons for the introduction of stochastic exponentials in the above Definition \ref{defisolu}. Even though we only use a special class of stochastic exponentials ($h$ is deterministic), their use may recall the classical Girsanov's Theorem. 
Indeed, if $h$ is cadlag, by Girsanov's Theorem the expected value $V= \E[uF]$ is the same as the expected value of the process $u$ under a new probability measure $\QQ$,
which has density $F_t$ (see Definition \ref{def Ft}) with respect to the reference probability measure $\PP$:
\begin{equation*}
\frac{d \QQ}{d \PP} \big|_{\calF_t} = F_t \, . 
\end{equation*}
This is because
\begin{equation*}
 \E^\PP \big[ u(t,x) F \big] =  \E^\PP \big[ u(t,x) F_t \big] = \E^\QQ [ u ]  \, .
\end{equation*}
From this point of view, one could interpret our approach 
to the uniqueness problem as follows: 
we show that the expected value of our solution $u$ is unique under a family of probability measures, 
which is large enough to ensure uniqueness of the solution (almost surely).
\end{remark}

\section{ Existence of weak solutions}
\label{EXISTENCE}

We shall here prove existence of quasiregular solutions, under hypothesis \ref{hyp}. The key hypothesis is \eqref{con2}, which allows to obtain a-priori estimates. The existence of weak solutions is then classical (see for example \cite{FGP2} or the discussion in \cite{BrisLion}), but we still have to check the regularity in mean of such solutions.

\begin{theorem}\label{lemmaexis1}  Under the conditions of Hypothesis \ref{hyp}, there exit quasiregular weak solutions $u$ of the Cauchy problem \eqref{trasport}.
\end{theorem}

\begin{proof} We divide the proof into two steps. First, using an approximation procedure we shall
 prove that the problem \eqref{trasport} admits weak solutions under our hypothesis. Then, in the second step, we will show that the solutions obtained as limit of regularized problems in the first step are
   indeed quasiregular solutions.
\bigskip

{\it Step 1: Existence.} Let $\{\rho_\varepsilon\}_\varepsilon$ be a family of standard symmetric mollifiers.
 Consider a  nonnegative smooth cut-off function $\eta$ supported on the ball of radius 2 and such that
 $\eta=1$ on the ball of radius 1. For every $\varepsilon>0$ introduce the rescaled functions
 $\eta_\varepsilon (\cdot) =  \eta(\varepsilon \cdot)$. Using these two families of functions we define the family of regularized coefficients as
 $b^{\epsilon}(t,x) = \eta_\varepsilon(x) \big( [ b(t,\cdot) \ast \rho_\varepsilon (\cdot) ] (x) \big) $.
  Similarly, define the family of regular approximations of the initial condition $u_0^\varepsilon (x) = 
   \eta_\varepsilon(x) \big( [ u_0(\cdot) \ast \rho_\varepsilon (\cdot) ] (x)  \big) $.

Remark that any element $b^{\varepsilon}$, $u_0^\varepsilon$, $\varepsilon>0$ of the two families
 we have defined is smooth (in space) and compactly supported, therefore with bounded
 derivatives of all orders. Then, for any fixed $\varepsilon>0$, the classical theory of Kunita, see \cite{Ku} or \cite{Ku3},
 provides the existence of a unique 
 solution $u^{\varepsilon}$ to the regularized equation 
\begin{equation}\label{trasport-reg}
 \left \{
\begin{aligned}
    &d u^\varepsilon (t, x,\omega) +  \nabla u^\varepsilon (t, x,\omega)  \cdot \big( b^\varepsilon (t, x)  dt +
 \circ d B_{t}(\omega) \big) = 0 \, ,
    \\[5pt]
    &u^\varepsilon \big|_{t=0}=  u_{0}^\varepsilon
\end{aligned}
\right .
\end{equation}
together with the representation formula
\begin{equation}\label{repr formula}
u^\varepsilon (t,x) = u_0^\varepsilon \big( (\phi_t^\varepsilon)^{-1} (x) \big)
\end{equation}
in terms of the (regularized) initial condition and the inverse flow $(\phi_t^\varepsilon)^{-1}$ associated to the equation of characteristics of \eqref{trasport-reg}, which reads 
\begin{equation*}
d X_t = b^\varepsilon (t, X_t) \, dt + d B_t \, ,  \hspace{1cm}   X_0 = x \,.
\end{equation*}
Moreover, the Jacobian of the flow solves pathwise the deterministic ODE (see \cite{Ku})
\begin{equation*}
d J \phi_t^\varepsilon (x,\omega) = \dive b^\varepsilon \big(  t, \phi_t^\varepsilon (x,\omega)  \big)  J \phi_t^\varepsilon (x,\omega) \, dt
\end{equation*}
and thus
\begin{equation*}
\log \big( J \phi_t^\varepsilon (x,\omega) \big) = \int_0^t \dive b^\varepsilon \big( s, \phi_s^\varepsilon (x,\omega)  \big) \,  ds \, .
\end{equation*}
Due to assumption \eqref{con2}, the Jacobian of the flow is therefore bounded uniformly in $\varepsilon$,
 because $\int_0^t \dive b^\varepsilon ds$ is. Then, we can use the random change of variables $(\phi_t^\varepsilon)^{-1}(x)  \mapsto x$ to obtain that almost surely
\begin{align}   \label{u_eps uniform}
\int_{\R^d}  \big|  u^{\epsilon}(t,x) \big|^{2}  dx &=  \int_{\R^d}  \big| u_{0}^{\epsilon}\big( (\phi_{t}^\varepsilon)^{-1} (x,\omega)\big) 
  \big|^{2}  dx = \int_{\R^d}  \big|u_{0}^{\epsilon}(x) \big|^2   J \phi_{t}^\varepsilon (x,\omega) \, dx    \nonumber  \\
  &\leq C \int_{\R^d}     \big|u_{0}^{\epsilon}(x)   \big|^{2}  dx \, .
\end{align}
\medskip

If $u^\varepsilon$ is a solution of \eqref{trasport-reg}, it is also a weak solution, which means that for any test function $\varphi \in C_c^\infty(\R^d)$, $u^\varepsilon$ satisfies 
the following equation (written in It\^o form)
\begin{equation} \label{transintegralR2}
\begin{aligned}
    \int_{\R^d} u^\varepsilon(t,x) &\varphi(x) \, dx= \int_{\R^d} u^\varepsilon_{0}(x) \varphi(x) \, dx
   +\int_{0}^{t} \!\! \int_{\R^d} u^\varepsilon(s,x) \, b^\varepsilon(s,x) \cdot \nabla \varphi(x) \, dx ds
\\[5pt]
   &+ \int_{0}^{t} \!\! \int_{\R^d} u^\varepsilon(s,x)\,  \dive \, b^\varepsilon(s,x) \, \varphi(x) \, dx ds 
\\[5pt]
    &+ \int_{0}^{t} \!\! \int_{\R^d} u^\varepsilon(s,x) \, \partial_{i} \varphi(x) \, dx \, dB^i_s
    + \frac{1}{2} \int_{0}^{t} \!\!\int_{\R^d} u^\varepsilon(s,x) \Delta \varphi(x) \, dx ds \, .
\end{aligned}
\end{equation}
To prove the existence of weak solutions to \eqref{trasport} we shall show that the sequence $u^\varepsilon$ admits a convergent subsequence, and pass to the limit in the above equation along this subsequence. This is done following the classical argument of \cite[Sect. II, Chapter 3]{Pardoux}, see also \cite[Theorem 15]{FGP2}.

\medskip
Let us denote by $\calY$ the separable metric space $C([0,T];  L^{2}(\R^d))$. Since $u_0^\varepsilon$
 is uniformly bounded in $L^2(\R^d)$, by \eqref{u_eps uniform}  $u^\varepsilon$ is also uniformly bounded in the spaces $L^{\infty}\big(\Omega; \calY \big)$ and $L^2 \big( \Omega \times [0,T] \times \R^d \big)$. 
 By the representation formula \eqref{repr formula} itself, we also get the uniform bound in 
$L^\infty  \big( \Omega \times [0,T] \times \R^d \big)$. Therefore, there exists a sequence 
$\varepsilon_n \to 0$ such that $u^{\varepsilon_n}$ weak-$\star$ converges in $L^\infty$ 
and weakly in $L^2$ to some process 
$u\in L^2  \big( \Omega \times [0,T] \times \R^d \big) \cap L^\infty  \big( \Omega \times [0,T] \times \R^d \big) $. To ease notation, let us denote $\varepsilon_n$ by $\varepsilon$ and for every $\varphi\in  C_c^{\infty}(\R^d)$, $\int_{\R^d} u^\varepsilon (t,x) \varphi(x) \, dx $ by $u^\varepsilon(\varphi)$, including the case $\varepsilon=0$.

\medskip
Clearly, along the convergent subsequence found above, the sequence of nonanticipative processes $u^{\varepsilon}(\varphi)$ also weakly converges in $L^2 \big( \Omega \times [0,T])$ to the process $u( \varphi)$, which is progressively measurable because the space of non\-anticipative processes is a closed subspace of $L^2 \big( \Omega \times [0,T])$, hence weakly closed. It follows that the It\^o integral of the bounded process $u(\varphi)$ is well defined. Moreover, the mapping $f\mapsto \int_0^\cdot f(s) \cdot dB_s$ is linear continuous from the space of nonanticipative $L^2(\Omega \times[0,T] ; \R^d)$-processes to $L^2(\Omega \times [0,T])$, hence weakly continuous. Therefore, the It\^o term $\int_0^\cdot u^\varepsilon(\nabla \varphi) \cdot \, dB_s$ in \eqref{transintegralR2} converges weakly in $L^2(\Omega\times [0,T])$ to $\int_0^\cdot u(\nabla \varphi) \cdot \, dB_s$.

\medskip
Note that the coefficients $b^\varepsilon$ and $\dive b^\varepsilon$ are strongly convergent in $L^2_\loc([0,T] \times \R^d)$ and $L^1\big([0,T]; L^\infty(\R^d) \big)$ respectively. This implies that $b^\varepsilon \cdot \nabla \varphi + \varphi \dive b^\varepsilon$ strongly converges in $L^1([0,T]; L^2(\R^d) )$ to $b \cdot \nabla \varphi + \varphi \dive b$ because $\varphi$ is of compact support. We can therefore pass to the limit also in all the remaining terms in \eqref{transintegralR2}, to find that the limit process $u$ is a weak solution of \eqref{trasport}.

\bigskip
{\it Step 2: Regularity.}  Consider now the equation of the regularized problem \eqref{trasport-reg} written in It\^o integral form. Fix any $F\in \calX$, multiply the equation by $F$ and take expectations. We have that $V_{\varepsilon}(t,x) =  \mathbb{E} [u^{\varepsilon}  (t,x) F ]$ satisfies
\begin{align*}
V_\varepsilon(t,x)  =& \, V_0^\varepsilon (x) - \int_0^t b^\varepsilon (s,x) \cdot \nabla V_\varepsilon(s,x) \, ds \\
& - \E \Big[ \int_0^t \nabla u_\varepsilon(s,x)  \cdot dB_s \ F \Big] + \frac{1}{2} \int_0^t  \Delta V_\varepsilon(s,x) \, ds \, ,
\end{align*}
where $V_0^\varepsilon (x) = u_0^\varepsilon (x)$ because exponential integrals have expected value $1$. Since $\forall \varepsilon>0$ $u_\varepsilon$ is regular, for every fixed $x\in \R^d$, $Y_s = \nabla u^\varepsilon (s, x)$ is an adapted, square integrable process. Using the properties of stochastic integrals we can then rewrite the expected value of the stochastic integral as the expected value of a Lebesgue integral against the function $h$. We leave the proof of this technical step to the Appendix, see Lemma \ref{lemma B-F}. This allows us to obtain a more convenient form of the above equation:
$$
\begin{aligned}
    V_{\varepsilon}(t,x)  &=  V_{0}^\varepsilon(x)
   -\int_{0}^{t} \! \nabla V_{\varepsilon}(s,x) \cdot \big( {b^\varepsilon}(s,x) + h(s) \big) \, ds
     + \frac{1}{2}\int_{0}^{t} \!    \Delta V_{\varepsilon}(s,x)   \, ds \,.
\end{aligned}
$$

Rewrite this in differential form (recall that $V_\varepsilon$ is regular, because $u^\varepsilon$ is) and multiply the equation by $2V_\varepsilon$:
\begin{align*}
\partial_t V_{\varepsilon}^2(t,x)  &=    - \nabla V_{\varepsilon}^2(s,x) \cdot \big( {b^\varepsilon}(s,x) + h(s) \big)  +  V_\varepsilon \Delta V_{\varepsilon}(s,x)    \,.
\end{align*}
Now, integrating in time and space we get
$$
\begin{aligned}
    \int_{\R^d} V_{\varepsilon}^{2}(t,x)   \, dx = & \,  \int_{\R^d} \big(V_{0}^\varepsilon \big)^{2}(x) \, dx    \\[5pt]
     &  -  \int_{0}^{t} \int_{\R^d} \! \nabla \big( V_{\varepsilon}^{2} \big) (s,x) \cdot \big( b^\varepsilon (s,x) + h(s) \big) \, dx  ds \\[5pt]
     & 	- \int_{0}^{t}   \int_{\R^d} \!    \big| \nabla V_{\varepsilon}(s,x) \big| ^{2} \, dx  ds \, ,
\end{aligned}
$$
and rearranging the terms conveniently we finally obtain the bound
\begin{align}  \label{Gronwall1}
 \int_{\R^d} V_{\varepsilon}^{2}(t,x)  \, dx \  + &  \int_0^t \int_{\R^d} \!    |\nabla V_{\varepsilon}(s,x) | ^{2} \, dx  
 ds  \nonumber \\[5pt]
		& =	\int_{\R^d} \big(V_{0}^\varepsilon \big)^{2}(x)\,  dx
   +\int_{0}^{t} \int_{\R^d} V_{\varepsilon}^{2}(s,x) \   \dive  {b^\varepsilon}(s,x) \, dx  ds   \nonumber
		\\[5pt] 
		&\le \int_{\R^d} (V_{0}^\varepsilon)^{2}(x) \, dx +    \int_{0}^{t} \gamma(t) \int_{\R^d} 
		V_{\varepsilon}^{2}(s,x) \, dx  ds    \, ,
\end{align}
for some function $\gamma\in L^1(0,T)$ which can be chosen independently of $\varepsilon$, because $\dive b^\varepsilon$ is uniformly bounded in $L^1\big( [0,T] ; L^\infty (\R^d) \big)$. We can now apply Gr\"onwall's Lemma and obtain
\begin{equation}   \label{un}
    \int_{\R^d} V_{\varepsilon}^{2}(t,x)  \, dx  \leq C
			\int_{\R^d} (V_{0}^\varepsilon)^{2}(x) \, dx \, ,
\end{equation}
where the constant $C$ can be chosen uniformly in $\varepsilon$ due to the integrability of $\dive b$.

Plugging \eqref{un} into \eqref{Gronwall1} we also get
\begin{equation}\label{dos}
\int_{0}^{t}   \int_{\R^d} \!    \big| \nabla V_{\epsilon}(s,x)  \big| ^{2}  \, dx  ds   \leq   C   \int_{\R^d} \big( V_{0}^\varepsilon \big)^{2}(x) \, dx \, .
\end{equation}

From (\ref{un}) and (\ref{dos}) we deduce the existence of a subsequence $\varepsilon_n$ 
(which can be extracted from the subsequence used in the previous step) for which  $ V_{\varepsilon_n}(t,x)$ converges weakly to the function $V(t,x)=\E[u(t,x)\, F]$ in $ \calY$ and such that $ \nabla V_{n}(t,x)$ converges weakly to $\nabla V(t,x)$ in $ L^{2}([0,T] \times \R^d )$. This allows us 
 to conclude that $V\in L^2([0,T]; H^{1}(\R^d) ) \cap C([0,T]; L^2(\R^d))$.
  \end{proof}

\section{Uniqueness}
\label{UNIQUE}

In this section, we shall present a uniqueness theorem
for the SPDE (\ref{trasport}). 
As in the by now classical setting, the proof is based on the commutator Lemma \ref{conmuting}. If applied in the usual way, this lemma requires to have $W^{1,1}$ regularity either for the drift coefficient $b$ or for the solution $u$. This is precisely what we want to avoid: in our setting we have neither of them, since we want to deal with possibly discontinuous solutions and drift coefficients. However, the key observation is that it is enough to ask such Sobolev regularity of the expected values $V(t,x)= \E [u(t,x) F]$ for $F\in \calX$, not on the solution $u$ itself.

Before stating and proving the main theorem of this section, we shall introduce some further notation and the key lemma on commutators. We stress that in this section we will be working under both the sets of Hypothesis \ref{hyp} and \ref{hyp2}.
\bigskip

\bigskip
Let  $\{\rho_{\varepsilon} \}$ be a family of standard positive symmetric mollifiers. Given two functions $f:\R^d \mapsto \R^d$ and $g:\R^d \mapsto \R$, the commutator $\calR_\varepsilon(f,g)$ is defined as
\begin{equation}\label{def commut}
    \mathcal{R}_{\varepsilon}(f,g):= (f \cdot \nabla ) (\rho_{\epsilon}\ast g )- \rho_{\varepsilon}\ast  (f\cdot \nabla g ) \, .
  \end{equation}
The following lemma is due to Le Bris and Lions \cite{BrisLion}.

\begin{lemma} \label{conmuting} (C. Le Bris - P. L.Lions )
Let  $f \in  L^2_\loc(\R^d ) $,  $g \in H^{1}(\R^d) $. Then, 
passing to the limit as $\varepsilon\rightarrow 0$
\[
    \mathcal{R}_{\varepsilon}(f,g) \rightarrow 0  \qquad in  \qquad L^1_\loc(\R^d) \, .
  \]
\end{lemma}
\bigskip

We can finally state our uniqueness result.
\bigskip

\begin{theorem}\label{uni} 
Under the conditions of Hypothesis \ref{hyp} and \ref{hyp2}, uniqueness holds for quasiregular weak solutions of  the Cauchy problem \eqref{trasport} in the following sense:
if $u,v \in L^2 \cap L^\infty \big(\Omega \times [0,T] \times \R^d \big)$ are two quasiregular weak solutions 
with the same initial data $u_{0}\in L^{2}(\mathbb{R}^{d}) \cap L^{\infty}(\mathbb{R}^{d})$, then  $u= v$ almost everywhere 
in $ \Omega  \times [0,T] \times \R^d $. 
\end{theorem}

\begin{proof} The proof is essentially based on energy-type estimates on $V$ (see equation \eqref{eqen} below) combined with Gr\"onwall's Lemma. 
However, to rigorously obtain \eqref{eqen} two preliminary technical steps of regularization and localization 
are needed, where the above Lemma \ref{conmuting} will be used to deal with the commutators appearing 
in the regularization process.  \medskip

{\it Step 0: Set of solutions.} Remark that the set of quasiregular weak solutions is a linear subspace of $L^2 \big(\Omega \times [0,T] \times \R^d \big)$, because the stochastic transport equation is linear, and the regularity conditions is a linear constraint. Therefore, it is enough to show that a quasiregular weak
solution $u$ with initial condition $u_0= 0$ vanishes identically. \bigskip

{\it Step 1: Smoothing.}
Let $\{\rho_{\varepsilon}(x)\}_\varepsilon$ be a family of standard symmetric mollifiers. For any $\varepsilon>0$ and $x\in\R^d$ we can use $\rho_\varepsilon(x-\cdot)$ as test function in the equation \eqref{eq V} for $V$. Observe that considering only quasiregular weak solutions starting from $u_0=0$ results in $V_0=0$. Using the regularity of $V$, we get
$$
\begin{aligned}
      \int_{\R^d} V(t,y) \rho_\varepsilon(x-y) \, dy  = &\, - \int_{0}^{t}  \int_{\R^d} \big( b(s,y) \cdot \nabla V(s,y)  \big)  \rho_\varepsilon(x-y) \ dy ds    \\[5pt]
        &-  \int_{0}^{t} \int_{\R^d} \! \big( h(s) \cdot \nabla V(s,x) \big) \rho_\varepsilon(x-y) \, dy ds     \\[5pt]
    & - \frac{1}{2}\int_{0}^{t}   \int_{\R^d} \! \nabla V(s,y) \cdot  \nabla_y \, \rho_\varepsilon(x-y) \, dy  ds \,.
\end{aligned}
$$
Set $V_\varepsilon(t,x)= V(t,x) \ast_x \rho_\varepsilon(x)$. Using the definition \eqref{def commut} of the commutator $\big(\calR(f,g)\big) (s)$ with $f=b(s, \cdot)$ and $g=V(s, \cdot)$, we have for each $t \in [0,T]$
$$
\begin{aligned}
    V_{\varepsilon}(t,x) + \int_{0}^{t} \big( b(s,x) + h(s) \big) \cdot  \nabla V_{\varepsilon}(s,x) \,  ds   & -    \frac{1}{2}\int_{0}^{t}   \Delta V_{\varepsilon}(s,x) \, ds      \\[5pt]
    & =         \int_{0}^{t} \big(\mathcal{R}_{\epsilon}(b,V) \big) (s) \,  ds  \, .
\end{aligned}
$$
By the regularity of $b$ and $V$, provided by \eqref{con1} and the Definition of solution \ref{defisolu}, one easily obtains that $\mathcal{R}_{\epsilon}(b,V) \in L^1([0,T]) $. Therefore, $V_\varepsilon$ is differentiable in time. To obtain an equation for $V_\varepsilon^2$ we can differentiate the above equation in time, multiply by $2V_\epsilon$ and integrate again. We end up with
\begin{equation}
\label{1000}
\begin{aligned}
    V_{\varepsilon}^{2}(t,x) + \int_{0}^{t} \big( b(s,x) + h(s) \big)  \cdot  \nabla \big( V_{\varepsilon}^{2}(s,x) \big) \, ds
-  \int_{0}^{t}    V_\varepsilon (s,x)  \Delta V_{\varepsilon} (s,x) \, ds
\\[5pt]
	=	 2 \int_{0}^{t}  V_{\varepsilon}(s,x)   \mathcal{R}_{\epsilon}(b,V) \, ds  \, .
\end{aligned}
\end{equation}
Remark that, by definition of solution, $V$ is bounded. Therefore, $V_\varepsilon$ is uniformly bounded. It follows that all the terms above have the right integrability properties, and the equation is well-defined.

\bigskip
{\it Step 2: Localization.} Consider a  nonnegative smooth cut-off function $\eta$ supported on the ball of
 radius 2 and such that  $\eta=1$ on the ball of radius 1. For any $R>0$ introduce the rescaled functions 
$\eta_R (\cdot) =  \eta(\frac{.}{R})$. Multiplying \eqref{1000} by $\eta_R$ and integrating over $ \R^d $
 we have 
$$
\begin{aligned}
   &  \int_{\R^d}   V_{\varepsilon}^{2}(t,x)   \eta_R(x)  \, dx  + \int_{0}^{t}  \int_{\R^d}  \big(  b(s,x) + h(s) \big)  \cdot  \nabla \big( V_{\varepsilon}^{2}(s,x) \big)  \eta_R(x) \, dx ds
\\[5pt]
       +& \int_{0}^{t}   \int_{\R^d}  |\nabla V_{\varepsilon}(s,x)|^{2}   \eta_R(x) \, dx  ds + \int_{0}^{t}   \int_{\R^d}  V_\varepsilon (t,x) \big( \nabla V_\varepsilon (t,x) \cdot  \nabla \eta_R(x) \big)  dx ds   \\[5pt]
    & \hspace{2cm} =  2 \int_{0}^{t}  \int_{\R^d}    V_{\varepsilon}(s,x)   \mathcal{R}_{\epsilon}(b,V)    \eta_R(x) \, dx   ds   \,  ,
\end{aligned}
$$
which we rewrite as
\begin{align} \label{eq V_eps^2}
   & \int_{\R^d}   V_{\varepsilon}^{2}(t,x)   \eta_R(x)  \,  dx      \nonumber   \\[5pt]
     & -\int_{0}^{t}  \int_{\R^d}  V_{\varepsilon}^{2}(s,x) \Big[  \big( b(s,x) + h(s) \big) \cdot \nabla \eta_R(x) +  \eta_R(x) \dive  b(s,x) \Big] \, dx ds       \nonumber  \\[5pt]
     & + \int_{0}^{t}   \int_{\R^d}  |\nabla V_{\varepsilon}(s,x)|^{2}   \eta_R(x) \, dx  ds  + \int_{0}^{t}   \int_{\R^d}  V_\varepsilon (t,x) \big( \nabla V_\varepsilon (t,x) \cdot \nabla \eta_R(x) \big)  dx ds   \nonumber   \\[5pt]
	 & \hspace{4cm}   =   2 \int_{0}^{t}  \int_{\R^d}    V_{\varepsilon}(s,x)   \mathcal{R}_{\epsilon}(b,V)    \eta_R(x) \, dx  ds   \, .
\end{align}

\bigskip
{\it Step 3: Passage to the limit.} Finally, in this step we shall pass to the limit in $\varepsilon$ and $R$ 
 to obtain uniqueness. 

Recall that due to \eqref{repr formula} $u$ is bounded, so that $V$ and $V_\varepsilon$ are (uniformly) bounded too. We first take the limit $\varepsilon\rightarrow 0$ in the above equation \eqref{eq V_eps^2}. By standard properties of mollifiers $V_\varepsilon \to V$ strongly in $L^{2}\big([0,T]; H^1(\R^d)\big) \cap C([0,T] ; L^2 (\R^d) )$, and we can use
 Lemma \ref{conmuting} and the uniform boundedness of $V_\varepsilon$ to deal with the
  term on the right hand side. We get
\begin{equation}\label{eqen}
\begin{aligned}
    \int_{\R^d}   V^{2}(t,x)   \eta_R(x) \, dx \,  + & \int_{0}^{t}   \int_{\R^d}  |\nabla V(s,x)|^{2}   \eta_R(x) \, dx  ds
\\[5pt]
+& \int_{0}^{t}   \int_{\R^d}  V (t,x) \big( \nabla V (t,x) \cdot \nabla \eta_R(x) \big)  dx ds \\
	&=\int_{0}^{t}  \int_{\R^d}  V^{2}(s,x) \big( b(s,x) + h(s) \big) \cdot  \nabla \eta_R(x) \,  dx   ds\\[5pt]
	&	\quad   +\int_{0}^{t}  \int_{\R^d}  \dive  b(s,x)   V^{2}(s,x)  \eta_R(x) \,  ds   \, .
\end{aligned}
\end{equation}
Using \eqref{con3} and the definition of $\eta_R$ we can now get rid of the first term on the right hand
 side by taking the limit $R\to\infty$. Indeed, for $R\ge 1$ we have that
\begin{align*}
(b+h)\cdot \nabla \eta_R & \le \big( |b| + |h| \big) \frac{\|\nabla \eta \|_\infty}{R}  \indicator{[R,2R]} \\
& \le 3 \|\nabla \eta \|_\infty \Big( \frac{|b|}{1+|x|} + \frac{|h|}{3R} \Big) \indicator{[R,2R]}
 \end{align*} 
 is bounded in $L^1\big( [0,T] ; L^\infty(\R^d)\big)$. Moreover, by definition of quasiregular weak solution we have that $V \in L^\infty\big( [0,T] ; L^2(\R^d) \big)$, and since the domain of integration (the support of $\nabla \eta_R$) leaves any compact as $R\to \infty$, we even have that $V \indicator{[R,2R]}$ goes to zero in $L^\infty\big( [0,T] ; L^2(\R^d) \big)$. Therefore,
$$
     \lim_{R \to \infty}	\int_{0}^{t}  \int_{\R^d}   V^{2}(s,x) \big(  b(s,x) +h(s) \big) \cdot  \nabla \eta_R(x) \,  dx   ds  = 0 \, .
$$
Likewise, since  $\nabla V \cdot \nabla \eta_R$ goes to zero in $L^\infty\big( [0,T]; L^2(\R^d) \big) $, also the last term of the left hand side of \eqref{eqen} goes to zero. 
We are left with
$$
\begin{aligned}
    \int_{\R^d}   V^{2}(t,x) \,  dx  &  +     \int_{0}^{t}   \int_{\R^d}  |\nabla V(s,x)|^{2} \,  dx  ds    =\int_{0}^{t}  \int_{\R^d}   \dive  b(s,x)   V^{2}(s,x)\,  dx  ds   \, .
\end{aligned}
$$
 By condition \eqref{con2}, we may write 
$$
     \int_{\R^d}   V^{2}(t,x)  \, dx  \leq  \int_{0}^{t}  \gamma(s)  \int_{\R^d}     V^{2}(s,x)  \, dx ds
$$
for some function $\gamma\in L^1(0,T)$. Applying Gr\"onwall's Lemma we conclude that for every $t\in[0,T]$, $V(t,x)=\E[u(t,x) F]=0$ for almost every $x\in\R^d$ and every $F\in \calX$. 

\bigskip
{\it Step 4: Conclusion.} From the result of the previous step we get that $\int_{[0,T]\times\R^d} \E [u(t,x) F] f(t,x)\, dx dt =0$ for all $F\in \calX$ and $f\in C^\infty_c ([0,T] \times \R^d)$. By linearity of the integral and the expected value we also have that
\begin{equation}\label{eq uY}
\int_{\R^d} \E\big[ u(t,x) \,Y \big] f(t,x) \, dx dt =0
\end{equation}
for every random variable $Y$ which can be written as a linear combination of a finite number of $F\in \calX$. Since by Lemma \ref{expo} the span generated by $\calX$ is dense in $L^2(\Omega)$, \eqref{eq uY} holds for any $Y\in L^2(\Omega)$. Linear combinations of products of functions $Y f(t,x)$ are dense in the space of test functions $\psi(\omega,x,t) \in L^2(\Omega  \times [0,T]\times \R^d)$, so that
\begin{equation*}
\int_{[0,T]\times \R^d} \E \big[ u(t,x) \,\psi(\omega,t,x) \big]  \, dxdt =0  \, ,
\end{equation*}
and $u=0$ almost everywhere on $\Omega\times [0,T] \times \R^d$. 
\end{proof}

\begin{remark}
Since our solutions to the SPDE \eqref{trasport} are only integrable, we cannot expect to obtain an uniqueness result stronger than {\it ``almost everywhere"}. However, as soon as the solution $u$ is integrated against a test function in space ($u(\varphi)$ with the notation of Section \ref{EXISTENCE}) or in $\omega$ ($V$), we obtain a function which is continuous in time. Therefore, one can obtain that for any $\varphi\in C^\infty_c(\R^d)$, $u(\varphi)=0$ almost surely for all $t\in [0,T]$, or that for any $F\in \calX$, $V=\E[uF]=0$ for almost every $x\in \R^d$, for all $t\in [0,T]$.
\end{remark}

\begin{remark} 
\label{REMAS}  From the proof of Theorem \ref{lemmaexis1} it is possible to see that, under our weak hypothesis, {\bf any} weak solution $u$ of the Cauchy problem \eqref{trasport} which is the $L^\infty\big( \Omega ; \calY \big)$-limit of weak solutions to regularized problems has the regularity of a quasiregular 
weak solution, and is therefore unique by Theorem \ref{uni}. 
In other words, we have also proved uniqueness in the sense of Theorem \ref{uni} 
in the class of solutions which are limit of regularized problems.
\end{remark}

\section{Application to Stochastic Muskat Problem}
\label{SMP}

In this section we give an important motivation
of the theory developed in the previous sections, where the
uniqueness result in the class of quasiregular  solutions can 
be used to establish 
existence of solutions to the Stochastic Muskat 
problem. The model considered here is a stochastic 
generalization of the 
original Muskat problem, which was proposed in 1934 by 
Muskat \cite{MUSKAT} to study from Darcy's law
the encroachment of water into an oil sand. 
In fact, the model follows the main ideas in \cite{NCWN2},
with a (small) Brownian noise perturbation of the continuity equation. 
Note that with our choice of random perturbation we do not change the hyperbolic type condition of the 
original continuity equation, and further maintain the original structure of the Darcy's law for the
velocity vector field. 

\medskip
The Muskat problem is a piston--like displacement of two
immiscible fluids in a porous media. We will use the subscripts $o$ and  $w$, to distinguish between each phase of the mixture. 
We shall assume that, there exist 
a domain $U \subset \mathbb{R}^{d}$ occupied by the fluids,
which are separated by an
unknown (free) surface $S$ of
co--dimension one. 
Under the assumption that the fluids are immiscible, 
the domain $U$ is given by the union of two disjoint sets $U _{o}(t)$ and 
$U_{w}(t)$, with the common surface $S(t)$. 
For each $t\geqslant 0$ and $%
x\in U$ the mixture density $\rho (t,x)$ is given by 
\begin{align}
\rho (t,x)= \rho_w(t,x) \, 1_{U_w(t)}(x) +
\rho_o(t,x) \, 1_{U_o(t)}(x),  \label{RHO}
\end{align}
also the velocity $\mathbf{v}(t,x)$ of the mixture is 
$$
   \mathbf{v}(t,x)=\mathbf{v}_{w}(t,x)\;1_{U _{w}(t)}(x)
   +\mathbf{v}_{o}(t,x)\;1_{U _{o}(t)}(x),
$$
where $\mathbf{v}_{\iota }(t,x)$, $(\iota =o,w)$, is the velocity
field of each component.

\medskip 
Let $\sigma> 0$ be a (small) parameter. 
Then we consider a stochastic balance of mass (also
called stochastic intrinsic continuity equation)
\begin{equation}
\partial _{t}\rho (t,x,\omega)+\mathrm{div} \Big(\rho (t,x,\omega)\; 
\big(\mathbf{v}(t,x) + \sigma \frac{d B_{t}}{dt}(\omega) \big) \Big)= 0.  \label{CEQDF}
\end{equation}
Moreover, from the
continuity of the normal velocities of the fluids on $S(t)$ and the
assumption that the fluids are incompressible, 
it follows that 
\begin{equation}
\mathrm{div} \, \mathbf{v}(t,x)=0.  \label{di}
\end{equation}

Now, let us consider the Conservation of Linear Momentum,
which follows from the Darcy's law (empirical) equation.
Denoting by $\mathbf{g}$ a given body force density,
the evolution of the velocity $\mathbf{v}$ is described by 
\begin{align}
 h(t,x,\rho \nu)\;\mathbf{v}(t,x)
&=- {\nabla }p(t,x)+ \mathbf{G}(t,x),  \notag \\
\mathbf{G}(t,x)&= \mathbb{E} \big[ \rho \, \mathbf{g} \big],
 \label{TMPV}
\end{align}%
where $h \geq h_0>0$ (takes into account the properties of the porous medium),
and the kinematic viscosity of the fluid mixture 
\begin{equation*}
\nu (t,x)=\nu _{w}(t,x)\,1_{U _{w}(t)}(x%
)+\nu _{o}(t,x)\,1_{U _{o}(t)}(x)
\end{equation*}%
is governed, analogously to the density, by the stochastic transport equation 
\begin{equation}
\partial _{t}\nu (t,x,\omega)+\mathrm{div}\Big(\nu (t,x,\omega)\;\mathbf{%
\big(v}(t,x)  + \sigma \frac{d B_{t}}{dt} \big) \Big)=0. \label{VEQDF}
\end{equation}%
The function $p \geq 0$ is
called the pressure of the fluid mixture, which
is defined by a similar formula to \eqref{RHO}. 
Since the function $h$ represents the medium properties,
and it is usually provided by mean values of small parts of the porous media, we consider
\begin{equation}
\label{HE}
   H(t,x)= \mathbb{E} \big[ h(t,x,\mu) \big] .
\end{equation}

\medskip
Finally, from equations \eqref{CEQDF}--\eqref{HE}, 
we introduce the stochastic 
Muskat type problem, denoted $\mathbf{SMP}$:
Let $T>0$ be any real number and $U  \subset \mathbb{R}%
^{d}$ $(d=2$ or $3)$ be an open and bounded domain having
smooth boundary.
For each $(t,x)\in U _{T}=[0,T] \times U$, find $(\rho (t,x,\omega),\nu (t,%
x,\omega),\mathbf{v}(t,x))$ solution of 
\begin{equation}
\label{SSMP}
\left\{ \begin{aligned} &\partial_t \rho 
+ \Big(\mathbf{v} + \sigma  \frac{d B_{t}}{dt} \Big) \cdot{\nabla} \, \rho=
0 \, , \qquad
\partial_t \nu  + \Big(\mathbf{v} + \sigma \frac{d B_{t}}{dt} \Big) \cdot {\nabla}\, \nu= 0 \, , \\[5pt]
& H \; \mathbf{v}= - {\nabla} p + \mathbf{G} \,, \quad {\rm
div}\, \mathbf{v}= 0 \,, 
\\[5pt] &\rho|_{t=0}= \rho_0 \, , \quad
\nu|_{t=0}= \nu_0\, , \quad  (\mathbf{v} \cdot \mathbf{n}) \big|_{\Gamma_T} = 0 \, ,
\end{aligned}\right.
\end{equation}
where $\rho _{0}$, 
$\nu _{0}$ are given initial data for the density and
viscosity respectively, and $\mathbf{n}$ 
is the unitary normal field to $\Gamma_T= [0,T] \times \partial U$. 
One recalls that, by Theorem 1.2 of \ \cite{Temam} for any vector field 
$\mathbf{v} \in L^{2}$, 
satisfying $\mathrm{div} \mathbf{v}= 0$ in
the distribution sense,  
the normal component of $\mathbf{v}$ i.e. 
$\mathbf{v}_{\mathbf{n}}:= \mathbf{v} \cdot \mathbf{n}$, exists and
belongs to $H^{-1/2}$.

The next definition tells us in which sense a triple $(\rho,\nu,\mathbf{v})$ is a weak solution of 
\eqref{SSMP}.

\begin{definition}
\label{defisoluSMP}  
Given $\rho_0, \nu_0 \in L^\infty(U)$, 
a triple $(\rho,\nu,\mathbf{v})$ is called a weak solution to $\mathbf{SMP}$, 
if $\rho, \nu \in L^\infty \big( \Omega \times[0,T] \times U \big)$ are
stochastic processes, and $\mathbf{v} \in L^2((0,T) \times U)$ satisfy:

\begin{itemize}
\item For any $\varphi \in C_c^{\infty}(\R^d)$, the real valued processes 
$\int \rho(t) \varphi dx$,  $\int \nu(t) \varphi dx$, have continuous modification which are
$\mathcal{F}_{t}$-semimartingale, and for all $t \in [0,T]$, we have $\mathbb{P}$-almost surely

\begin{equation}
\label{DISTRHO}
\begin{aligned}
    \int_{U} \rho(t) \varphi \ dx = &\int_{U} \rho_{0} \varphi \ dx
   +\int_{0}^{t} \!\! \int_{U} \rho(s) \ \mathbf{v}^i(s) \partial_{i} \varphi \ dx ds
\\[5pt]
    &+ \int_{0}^{t} \!\! \int_{U} \rho(s) \ \partial_{i} \varphi \ dx \, {\circ}{dB^i_s}, 
\end{aligned}
\end{equation}

\begin{equation}
\label{DISTNU}
\begin{aligned}
    \int_{U} \nu(t) \varphi \ dx = &\int_{U} \nu_{0} \varphi \ dx
   +\int_{0}^{t} \!\! \int_{U} \nu(s) \ \mathbf{v}^i(s) \partial_{i} \varphi \ dx ds
\\[5pt]
    &+ \int_{0}^{t} \!\! \int_{U} \nu(s) \ \partial_{i} \varphi \ dx \, {\circ}{dB^i_s} \, .
\end{aligned}
\end{equation}

\item For each test function $\boldsymbol{\psi} \in \mathbf{V}(U)$
\begin{equation}
   \int_U H(t) \mathbf{v}(t) \cdot \boldsymbol{\psi} \ dx= \int_U \mathbf{G}(t) \cdot \boldsymbol{\psi} \, dx,
\end{equation}
where $\mathbf{V}(U):= \{\boldsymbol{\psi} \in L^2(U): \dive \boldsymbol{\psi}= 0 \; \text{in} \; \clg{D}'(U), 
\boldsymbol{\psi} \cdot \mathbf{n}= 0 \; \text{on} \; \partial U \}$.
\end{itemize}
\end{definition}

The solution of this problem is still open, 
and we leave this labor for future research.
But we believe that our contribution, 
providing a well-posedness result for the stochastic transport equations 
under very weak hypothesis on the drift term, is a first essential step 
towards the solution of the SMP.
Indeed, one way to solve it is trying 
to apply Schauder's fixed point argument.
Let us give the main idea:

First, we may consider 
$\underline{M}= \min \{\|\rho\|_{\infty},\|\nu\|_{\infty} \}$, $\overline{M}= \max \{\|\rho\|_{\infty},\|\nu\|_{\infty} \}$,
and define the closed convex subset 
\begin{equation}
\mathcal{Z}:=\{(\rho ,\nu ) \in L^{2}(\Omega \times [0,T] \times U)^{2}:
\rho,\nu \in \lbrack \underline{M}, \overline{M} \rbrack \quad \text{a.e.}\}
\label{eqt2.5}
\end{equation}%
of the Banach space $L^{2}(\Omega \times [0,T] \times U)^{2}$, with the norm 
\begin{equation*}
    ||(\rho ,\nu )||_{L^{2}}\!\!
    :=||\rho ||_{L^{2}}\!\!+||\nu ||_{L^{2}}.
\end{equation*}

Now, let $(\overline{\rho },\overline{\nu })$ be an arbitrary fixed
element of $\mathcal{Z}$, and consider for $\sigma> 0$ the coupled system 
\begin{equation}
\begin{cases}
     \mathbb{E} [h(t,x,\overline{\rho }\ \overline{\nu })] \mathbf{v}
     = -\nabla p +  \mathbb{E} [\overline{\rho } \mathbf{g}],
     \quad \mathrm{div}\mathbf{v}=0\quad \text{in } U, 
     \\ 
     \mathbf{v} \cdot \mathbf{n}= 0 \quad \text{on } \partial U,
\end{cases}
\label{a}
\end{equation}%
and
\begin{equation}
\begin{cases}
    \partial _{t}\rho +\mathrm{div}((\mathbf{v} + \sigma \frac{dB_t}{dt}) \ \rho )=0, 
\\ 
\rho |_{t=0}=\rho _{0}, 
\end{cases}%
\quad 
\begin{cases}
\partial _{t}\nu +\mathrm{div}((\mathbf{v} + \sigma \frac{dB_t}{dt}) \ \nu )= 0,
\\  
\nu |_{t=0}=\nu _{0}.
\end{cases}
\label{b}
\end{equation}

Due to $\overline{\rho }$ and $\overline{\nu}$ regularities, it follows that we just have 
$\mathbf{v} \in L^2 \big([0,T] \times U  \big)$ as solution of \eqref{a}. 
Albeit, since the domain $U$ is bounded and taking into account the incompressibility condition \eqref{di}, 
we see that $\mathbf{v}$ satisfies both Hypothesis \ref{hyp} and \ref{hyp2}.  
Then, applying the well-posedness theory established in 
the preceding sections for stochastic transport equations, 
it is not difficult to show the solvability result for this system as presented in the following

\begin{lemma}
\label{PPA} For each $(\overline{\rho },\overline{\nu })\in \mathcal{Z}$,
there exists a unique solution $(\rho ,\nu ,\mathbf{v})$ \ of system %
\eqref{a}--\eqref{b}, such that%
\begin{equation}
(\rho ,\nu )\in \mathcal{Z},\qquad \Vert \mathbf{v}\Vert _{L^{2}((0,T) \times U)} \leqslant C,  \label{EUU}
\end{equation}%
where $C\geq 0$ is a positive constant depending only on data.
\end{lemma}

\bigskip 
One observes that, solving \eqref{a}--\eqref{b} 
we have constructed the operator 
\begin{equation*}
    P:\mathcal{Z}\rightarrow \mathcal{Z},\quad \mathbf{(}\rho ,\nu )=P(\overline{%
    \rho }\mathrm{,}\overline{\nu }),\;\quad \forall \,(\overline{\rho },%
    \overline{\nu })\in \mathcal{Z}.
\end{equation*}%
One then could use Schauder's theorem to find a fixed point of $P$, which will be a
weak solution of the system $\mathbf{SMP}$. To do so one has to show that 
$P(\mathcal{Z})$ is a relatively compact subset of the Banach space 
$L^{2}(\Omega \times [0,T] \times U)^{2}$, and also that the operator $P$ 
is continuous with respect to the norm $
\|(\cdot ,\cdot )\|_{L^{2}}$. Consequently, these are 
the two main steps to be done. 

\appendix
\section{Appendix}

\begin{definition}\label{def Ft}
Given a filtered probability space with an $\R^d$-valued Brownian motion defined on it, 
$( \Omega, \calF, P, \calF_t, B_t)$, for any $h \in L^2([0,T] ;\R^d)$, we can define the random process
$$
 F_t= \exp \Big( \int_0^t h(s) \cdot d B_s - \frac{1}{2} \int_0^t  |h(s)|^2 \, ds \Big) \, ,
$$
for $t\in[0,T]$. Such random processes are called stochastic exponentials.
\end{definition}

\medskip
We recall that stochastic exponentials satisfy the following SDE ( see \cite[proof of Theorem 4.3.3]{oksen} ) 
\begin{equation}\label{SDEexpon2}
F_t =   1 +  \int_{0}^{t}  h(s) F_s \, dB_s \, .
\end{equation}
This can be obtained applying It\^o formula to $F_t$. By Novikov's condition it also follows that any stochastic exponential $F_t$ is an $\calF_t$-martingale, and $\E[F_t] = 1$.

\medskip
When $t=T$, we shall use the short notation $F=F_T$ and, with a slight abuse of notation, 
still call the random variable $F$ a stochastic exponential. Let us recall the definition of the following space of random variables, which we call the space of stochastic exponentials:
$$ 
  \calX := \Big\{ F= \exp \Big( \int_0^T h(s) \cdot d B_s - \frac{1}{2} \int_0^T |h(s)|^2 \, ds \Big)   \    \Big|     \    h \in L^2 \big( [0,T] ;\R^d \big) \Big\} \, .
$$

\begin{remark}
Even though it is not really essential for our proof, we point out that for every $F \in \calX$ there exists a unique $h\in L^2(0,T)$ such that $F$ is the stochastic exponential of $h$. This can be easily shown using It\^o isometry.
\end{remark}

The following result, see \cite[Lemma 4.3.2]{oksen} or   \cite[Lemma 2.3]{Mao}, is a key fact for our analysis. Recall that $\calF=\calF_T$.

\begin{lemma}\label{expo} 
The span generate by $ \calX $ is a dense subset of $ L^{2}( \Omega)$.
\end{lemma}

We also have the following result.

\begin{lemma}\label{lemma B-F}
Let $F$ be a stochastic exponential and $Y_s\in L^2\big(\Omega\times [0,T]\big)$ an $\R^d$-valued, square integrable adapted process. Then,
\begin{equation}\label{B-F}
\E \Big[\int_{0}^{t}  Y_s \cdot d B_s \ F  \Big] =  \int_0^t h(s) \cdot \E \big[ Y_s  \ F   \big]  \, ds \, .
\end{equation}
\end{lemma}

\begin{proof}
Using the representation formula \eqref{SDEexpon2} we have
\begin{align*}
\E \Big[\int_{0}^{t}  Y_s \cdot d B_s \ F  \Big] &= \E \Big[\int_{0}^{t}  Y_s \cdot d B_s  \Big] +  \E \Big[\int_{0}^{t}  Y_s \cdot d B_s \ \int_0^T h(s) F_s \cdot dB_s  \Big] \\
&=  \E \Big[\int_{0}^{t}  Y_s \cdot h(s) F_s \, ds  \Big] \, .
\end{align*}
Since that $Y_s$ is $\mathcal{F}_{s}$-adapted, we obtain  
$$
\E \big[ Y_s   \, F_s  \big]  =  \E \big[ Y_s   \, F  \big] \, ,
$$
and \eqref{B-F} follows.
\end{proof}

\section*{Acknowledgements}

The author Ennio Fedrizzi is supported by the LABEX MILYON (ANR-10-LABX-0070) of Universit\'e de Lyon, within the program "Investissements d'Avenir" (ANR-11-IDEX-0007) operated by the French National Research Agency (ANR). 
The author Wladimir Neves is partially supported by
CNPq through the grants 
484529/2013-7, 308652/2013-4,
and also by FAPESP through the grant 2013/15795-9. 
Christian Olivera is partially supported by  
FAPESP 2012/18739-0 and   CNPq
through the grant 460713/2014-0.


\end{document}